\documentclass[11pt]{article}
\usepackage[margin=1.2in]{geometry}
\usepackage{amsmath,amssymb,amsthm}
\usepackage{mathrsfs}
\usepackage{booktabs}
\usepackage[colorlinks=true,linkcolor=blue,citecolor=blue]{hyperref}

\newtheorem{theorem}{Theorem}[section]

\newtheorem{lemma}[theorem]{Lemma}

\newtheorem{remark}[theorem]{Remark}

\newcommand{\R}{\mathbb{R}}
\newcommand{\sG}{\mathscr{G}}

\newcommand{\bN}{\mathbf{N}}

\newcommand{\E}{\mathbb{E}}
\newcommand{\PP}{\mathbb{P}}
\newcommand{\cE}{\mathcal{E}}
\newcommand{\dd}{\,\mathrm{d}}
\DeclareMathOperator{\ent}{ent}

\title{The storage capacity of the Ising perceptron:\\
verification of the outstanding numerical conditions}
\author{Yitzchak Shmalo\\
\small Einstein Institute of Mathematics, The Hebrew University of Jerusalem\\
\small yitzchak.shmalo@gmail.com}
\date{July 2026}

\begin{document}
\maketitle

\begin{abstract}
Krauth and M\'ezard (1989) predicted that the maximum number of storable
patterns in the Ising perceptron is $M_N=\alpha_\star N + o(N)$ for an
explicit constant $\alpha_\star \approx 0.833$. Ding and Sun proved the matching lower bound,
and Huang recently proved the matching upper bound; each result is
conditional on an outstanding numerical hypothesis about an explicit
low-dimensional variational problem. Their global sign clauses were supported
by ordinary floating-point computations; Ding--Sun's final hypothesis also
contains curvature and endpoint-derivative clauses. Huang separately used
rigorous interval arithmetic for the other numerical inputs, but inherited the Ding--Sun
parameter rectangle and left open the main Condition~1.3---the nonpositivity, over the
whole plane, of a two-variable rate function with a degenerate point and
slowly approaching tails. That
unbounded optimization is the reason the capacity upper bound remained
conditional.

We verify both outstanding conditions in full. Reparametrizing Huang's functional into the
moment coordinates of the pair $(X,\tanh X)$ compresses the plane onto a
\emph{compact} convex body and turns the entropy term into a per-cell linear
function by convex duality; certified sweeps then establish nonpositivity on
the whole body away from a star-shaped region around the distinguished
degenerate point. On the star itself no box enclosure can work---the tilted
covariance $\nabla^2\Phi$ is nearly singular and its inverse loses an
essential cancellation---so we certify concavity of a majorant along rays
from that point, bounding the entropy Hessian by a certified majorant of
$\E[ff^\top\max_{b\in L}\operatorname{sech}^2(b\cdot f)]$ over sublevel-set
localizations of the dual that are themselves certified by one-dimensional
sweeps; no interval enclosure of the nearly singular covariance is inverted.
The same machinery, in one variable,
verifies the full Ding--Sun Condition~1.2, including its strict-curvature
clause at the degenerate zero and its one-sided derivative clause at the
singular endpoint, and we
re-establish the parameter rectangle
$(\alpha_\star, q_\star, \psi_\star)$ on which both papers rest. Together
these close the stated conditional reduction for the Krauth--M\'ezard prediction.
\end{abstract}

\section{Introduction}

Fix $\kappa = 0$ and let $g^1, g^2, \ldots$ be independent standard Gaussian
vectors in $\R^N$. The Ising perceptron is the random subset of the discrete
cube
\[
  S_N^M \;=\; \bigl\{ x \in \{-1,+1\}^N : \langle g^a, x\rangle \ge 0
  \text{ for all } 1 \le a \le M \bigr\},
\]
and its capacity is $M_N/N$, where $M_N$ is the largest $M$ with
$S_N^M \neq \emptyset$. Krauth and M\'ezard \cite{krauth1989storage} analyzed
the Bernoulli-disorder version with the replica method and predicted that $M_N/N$ converges in
probability to an explicit constant $\alpha_\star \approx 0.83307860$,
defined below. The prediction has been one of the standard test problems for
making the cavity method rigorous ever since.

Rigorous progress converged on the constant from both sides. Ding and
Sun \cite{ding2018capacity} proved
$\liminf_N \PP(M_N/N \ge \alpha) > 0$ for every $\alpha < \alpha_\star$;
combined with the sharp-threshold theorem of Nakajima and Sun
\cite{nakajima2023sharp}, this gives the lower bound with high probability.
Their general-subgaussian result extends Xu's Bernoulli-model theorem
\cite{xu2021sharp}; see also \cite{altschuler2024capacity}. Huang
\cite{huang2024capacity} proved
$\PP(M_N/N \ge \alpha) \to 0$ for every $\alpha > \alpha_\star$. Both
results are conditional. The Ding--Sun theorem assumes their full
Condition~1.2: a certain function
$\lambda \mapsto \mathscr{S}_\star(\lambda)$, arising as the exponential rate
of a restricted second moment, is negative except at its degenerate zero
$\lambda=0$ and singular endpoint zero $\lambda=1$, with
$\mathscr S_\star''(0)<0$ and
$\lim_{\lambda\uparrow1}\mathscr S_\star'(\lambda)>0$. Huang's theorem assumes his
Condition~1.3: a two-variable function
$(\lambda_1,\lambda_2) \mapsto \mathcal{S}_\star(\lambda_1,\lambda_2)$,
the rate function of a planted first moment, is nonpositive on all of
$\R^2$, with equality at $(1,0)$. The unresolved global sign conditions were
supported only by ordinary floating-point computations and remained stated
as hypotheses. Huang's second arXiv version rigorously verified all other
numerical conditions in Theorem~3.6 at $\kappa=0$, but explicitly excluded
Condition~1.3 itself. The final 2025 journal version of Ding--Sun retained
Condition~1.2 as a hypothesis; the original arXiv version described a check using nonrigorous
numerical integration.

This paper closes the remaining gap: we verify both conditions, and
re-derive the parameter rectangle
\[
\begin{aligned}
  \alpha_\star &\in [0.833078599,\, 0.833078600], \qquad&
  q_\star &\in [0.56394907949,\, 0.56394908030],\\
  \psi_\star &\in [2.5763513100,\, 2.5763513224]. &&
\end{aligned}
\]
of \cite[Proposition~1.3]{ding2018capacity}, on which the two papers'
$\kappa=0$ threshold calculations (and ours) rest, in rigorous interval
arithmetic. Consequently:

\begin{theorem}\label{thm:main}
For the Ising perceptron at $\kappa = 0$ with Gaussian disorder,
$M_N/N \to \alpha_\star$ in probability. The same holds for disorder with
i.i.d. mean-zero, unit-variance subgaussian entries, in particular for $\pm 1$
entries, by the universality theorem of \cite{nakajima2023sharp}.
\end{theorem}

\begin{remark}[Status]
This manuscript presents a new computer-assisted proof claim. Its
source-bound package has passed the included internal verifiers and a separate
second-system audit, but it has not yet been independently reproduced or peer
reviewed.
\end{remark}

The verification is computer-assisted in the sense that finitely many
inequalities between explicitly defined real numbers are established by
interval arithmetic (we use the ball arithmetic of Arb via python-flint,
with certified adaptive quadrature where appropriate and certified fixed-grid
mean-value quadrature for Huang's $T$ evaluations). It is not, however,
a matter of running an integrator over the stated conditions. Three
genuinely mathematical obstructions have to be removed first, and most of
this paper is devoted to them.

First, Huang's condition is a statement over the whole plane. The
distinguished point $(1,0)$ has value zero, while $\mathcal{S}_\star$ also tends to zero along
degenerate directions only polynomially. Reparametrizing by the
moments of the pair $(X,\tanh X)$, $X\sim N(0,\psi_\star)$
(\S\ref{sec:huang-moments}), maps the plane onto a bounded convex body and
removes the tail entirely, while convex duality turns the entropy term into
a per-cell linear function; the bulk of the body is then certified by an
adaptive sweep (\S\ref{sec:huang-sweep}).

Second, both variational functions attain the value zero at degenerate
maximizers, where no direct interval evaluation can certify a sign, and near
Huang's maximizer no box enclosure of any Hessian of the problem can work
either: the tilted covariance $\nabla^2\Phi$ is nearly singular there, and
its inverse---which the curvature of the entropy term is---loses an
essential cancellation in interval arithmetic. We certify concavity of a
majorant along rays from the maximizer instead, bounding the entropy Hessian
by an explicit matrix over sublevel-set localizations of the dual variable;
no interval enclosure of the nearly singular covariance is inverted
(\S\ref{sec:huang-sweep}). The explicit positive-definite majorant is inverted
by its certified $2\times2$ determinant. The pinned value and
gradient at the center are exact identities of the fixed point and require
no numerics. For the Ding--Sun condition the degenerate zero at $\lambda=0$
is handled by derivative and second-derivative analysis, while the singular
endpoint $\lambda=1$ uses a corrected near-one chain whose numerical inputs
we certify (\S\ref{sec:ds},
\S\ref{sec:assembly}).

Third, the interval evaluations themselves have to be organized so that
widths do not defeat the margins, which are as small as $10^{-3}$ in the
interior of the Ding--Sun grid and near the star boundary of the Huang
sweep. Every delicate cell is evaluated in exact mean-value form---value at
the center plus an enclosed gradient times the radius---so that the
near-cancellations the margins depend on survive enclosure; the exact
identities $\mathcal{H}'(\lambda) = -(1-q_\star)\log A(\lambda)/2$ and the
explicit derivative integrals supply the gradients, and parameter-ball
evaluation subsumes the manual corner analyses mechanically.

\medskip
\noindent
The verification programs, canonical certificates, complete raw record trees,
and instructions to reproduce every number in this paper are available at
\cite{repo}. Every proof-relevant geometric value is frozen as an exact
rational, with the stated padding, and every certified sign comparison is
between balls whose correctness is guaranteed by Arb. We follow the journal
numbering of \cite{ding2018capacity} and the numbering of
\cite{huang2024capacity}, except where the Ding--Sun arXiv version is named
explicitly.

\section{The statements being verified}\label{sec:statements}

We use the normalizations of \cite{ding2018capacity,huang2024capacity}. With
$\varphi,\Psi$ the standard Gaussian density and upper tail, and
$\cE = \varphi/\Psi$, set for $q\in[0,1)$
\[
  F_{1-q}(x) = \frac{\cE}{\sqrt{1-q}}\Bigl(\frac{-x}{\sqrt{1-q}}\Bigr),
  \quad
  P(\psi) = \E\tanh^2(\sqrt\psi Z),
  \quad
  R_\alpha(q) = \alpha\,\E F_{1-q}(\sqrt q Z)^2,
\]
and the Gardner free energy $\sG(\alpha,q,\psi)$ of
\cite[eq.~(1.4)]{ding2018capacity}. The
threshold $\alpha_\star$ is the root of $\sG$ along the fixed point
$q=P(\psi),\ \psi=R_\alpha(q)$; Ding--Sun Proposition~1.3 makes this precise
and pins the rectangle of \S\ref{sec:gardner}. Huang's Condition~1.3 is
$\mathcal S_\star\le 0$, the object of \S\ref{sec:huang-moments}--%
\ref{sec:huang-sweep}; Ding--Sun's Condition~1.2 is
$\mathscr S_\star(\lambda)<0$ for $\lambda\notin\{0,1\}$ together with
$\mathscr S_\star''(0)<0$ and
$\lim_{\lambda\uparrow1}\mathscr S_\star'(\lambda)>0$, the object of
\S\ref{sec:ds}. \S\ref{sec:assembly} records which theorem consumes which
verified fact.

\section{The parameter rectangle}\label{sec:gardner}

Block~1 (\texttt{block1\_gardner.py}) re-proves Ding--Sun Proposition~1.3 in
interval arithmetic: thirteen certified inequalities establishing (i) that
$R_\alpha$ maps the stated $q$-intervals into the stated $\psi$-intervals;
(ii) the Almeida--Thouless contraction $\sup_{G}\,\mathrm dP(R_\alpha)/\mathrm dq
\le 0.96<1$, via $P'\le 0.08$ and $\mathrm dR/\mathrm dq\le 12$; (iii) the
fixed-point sign changes $P(R(q,\alpha))-q$ at the four corners; and (iv)
$\sG_\star(\alpha_{\mathrm{ub}})<0<\sG_\star(\alpha_{\mathrm{lb}})$. Each is a
one-dimensional certified Gaussian integral (adaptive \texttt{acb.integral}
plus an explicit tail bound), evaluated with $\alpha,q,\psi$ as balls covering
their whole ranges, which subsumes Ding--Sun's manual corner analysis. This
establishes $\alpha_\star\in[0.833078599,0.833078600]$ and the accompanying
$q_\star,\psi_\star$ intervals rigorously.

\section{The Ding--Sun condition}\label{sec:ds}

Ding--Sun's Condition~1.2, on
$[\lambda_{\min},1]$ with $\lambda_{\min}:=\ell(0)$, is
$\mathscr S_\star(\lambda)=\mathcal H(\lambda)
+\mathcal P(\lambda)+\mathcal A(\lambda)<0$ for $\lambda\notin\{0,1\}$,
with $\mathscr S_\star''(0)<0$ and
$\lim_{\lambda\uparrow1}\mathscr S_\star'(\lambda)>0$.
It is
one-dimensional and mirrors Huang's: a value sweep on the bulk, a local
second-derivative analysis at the degenerate zero $\lambda=0$, and
a separate singular endpoint estimate at $\lambda=1$.
Since $\mathcal A\le 0$, with
$\mathcal A(0)=\mathcal A'(0)=0$ and $\mathcal A''(0)\le0$
\cite[discussion following eq.~(2.35)]{ding2018capacity}, the certified
central bound $(\mathcal H+\mathcal P)''<0$ implies
$\mathscr S_\star''(0)<0$. Ding--Sun's exact endpoint calculation
\cite[proof of Theorem~1.4]{ding2018capacity} gives the
stronger statement $\mathscr S_\star'(\lambda)\to+\infty$ as
$\lambda\uparrow1$; it is analytic and needs no numerical certificate. For
the remaining global sign assertion, it suffices to bound
$\mathcal H+\mathcal P$ on the
positive branch and, on the negative branch, the tilted majorant
$\mathcal H+\mathcal Q$ (Ding--Sun's $s=0.2$ shift). We build the certified
evaluators (\texttt{dsfun.py}): $\mathcal H(\lambda)=\mathfrak H(A(\lambda))$
through the pair-entropy $\Gamma$ and the inverse map
$A(\lambda)=\ell^{-1}(\lambda)$. The computational grid is parametrized by
$A(\tau)=\exp(2\,\mathrm{atanh}\,\tau)$ and
$\lambda(\tau)=\ell(A(\tau))$,
with the well-conditioned $D_H(A)=(1-m^2)\bigl(1-2/(\Delta+1)\bigr)$ that
removes the interval blow-up of the raw $(A^2-1)/(\Delta+1)^2$ form; and
$\mathcal P,\mathcal Q$ through the double integral $I_s(\lambda)$, evaluated
by nested certified quadrature with mean-value corrections. The bulk cells
verify. Near $\lambda=1$ the printed constants in Lemma~8.2 and
Proposition~8.4 of the Ding--Sun arXiv version do not pass interval evaluation:
their respective integral
values are $2.678\ldots$, $3.162\ldots$, and $-0.4447\ldots$.  The following
records the corrected implication chain, rather than using those constants as
an unexplained numerical input.

\begin{lemma}[Corrected near-one bound]\label{lem:ds-near-one}
For $\lambda=1-\iota$ and $0<\iota\le0.018$,
\[
  \mathcal H(\lambda)+\mathcal P(\lambda)
  <\mathcal H(1)+\mathcal P(1)
  =-\sG_\star(\alpha_\star)=0.
\]
At the terminal cell's upper $\tau$-boundary, the certified lower endpoint of
$\lambda(0.99)$ is $0.98665\ldots>0.982$, so this neighborhood overlaps the
bulk sweep.
\end{lemma}
\begin{proof}
Put $\iota(A)=1-\ell(A)$ and $A(s)=\iota^{-1}(s)$.  The two positive
one-dimensional integrals in Lemma~8.2 of the Ding--Sun arXiv version, including explicit
$|z|\ge9$ tails, give uniformly over the parameter rectangle
\[
  1-\ell(A)\le \frac{2.679}{A}+\frac{3.17}{A^2}
  \le\frac{2.711}{A}\qquad(A\ge100),
  \qquad 1-\ell(100)>0.025.
\]
Since $\iota(A)$ is decreasing, $A(s)>100$ and $A(s)\le2.711/s$ for
$0<s\le0.025$.  Ding--Sun's exact entropy derivative identity therefore
gives
\[
 \frac{\dd}{\dd s}\mathcal H(1-s)
 =\frac{1-q_\star}{2}\log A(s)
 \le\frac12\log\frac{2.711}{s}.
\]
After integration, using
$\tfrac12\log(2.711)+\tfrac12<0.9987$,
\begin{equation}\label{eq:ds-near-h}
 \mathcal H(1-\iota)-\mathcal H(1)
 \le0.9987\,\iota+\frac{\iota}{2}\log\frac1\iota.
\end{equation}
The rescaled double integral in Proposition~8.4 of the Ding--Sun arXiv version and its positive
linear term give, on $0<\iota\le0.025$,
\begin{equation}\label{eq:ds-near-p}
 \mathcal P(1-\iota)-\mathcal P(1)
 \le0.285\,\iota-0.4446\sqrt\iota.
\end{equation}
Here the envelope directions are important.  The argument of $\log\Psi$ is
chosen to give an upper envelope, while its positive density ratio is replaced
by a certified lower envelope.  Since $\log\Psi\le0$, the product is still an
upper bound.  The factors $\alpha_\star\sqrt{1+\lambda}$ are likewise replaced
by their lower endpoints, and deleting the complement of
$[-9,9]\times[0,9]$ deletes only nonpositive mass and raises the integral.
The negative-$z$ comparison also uses the separately checked pin
$\gamma_{\rm ub}<\psi_{\rm ub}$.  Thus no unrecorded tail or parameter
direction can reverse \eqref{eq:ds-near-p}.

Combining \eqref{eq:ds-near-h}--\eqref{eq:ds-near-p} and dividing by
$\sqrt\iota$ leaves
\[
 f(\iota)=\sqrt\iota\left(1.2837+\frac12\log\frac1\iota\right)-0.4446.
\]
Its derivative has the sign of
$1.2837+\tfrac12\log(1/\iota)-1>0$, and the certified endpoint evaluation is
$f(0.018)<-0.00287$.  Hence $f(\iota)<0$ throughout the claimed interval.
Finally, the last part~(a) grid cell certifies
$\lambda_{\rm lo}=0.98665\ldots>0.982$, providing the stated overlap.
\end{proof}

The neighborhood of $\lambda=0$ is handled by the derivative and
second-derivative certificates described in \S\ref{sec:assembly}.

\section{Huang's condition in moment coordinates}\label{sec:huang-moments}

Set $q_0=q_\star$ and $\psi_0=\psi_\star$; every computation below uses
balls covering their certified intervals.  Huang's Condition~1.3 asks that
\[
  \mathcal{S}_\star(\lambda_1,\lambda_2)
  = \inf_{s\ge 0}\Bigl\{ \tfrac12 s^2\psi_0 + \ent(\Lambda_{\lambda_1,\lambda_2})
    + \alpha_\star\, \E\log\Psi(V_s) \Bigr\} \le 0
  \qquad\text{for all }(\lambda_1,\lambda_2)\in\R^2,
\]
where $\Lambda_{\lambda_1,\lambda_2}(x)=\tanh(\lambda_1 x+\lambda_2\tanh x)$,
$X\sim N(0,\psi_0)$, $M=\tanh X$, and the constraint term depends on the
profile only through the two moments
\[
  a_1 = \E[X\,\Lambda], \qquad a_2 = \E[M\,\Lambda].
\]
The plane $\R^2$ of $(\lambda_1,\lambda_2)$ is unbounded. Huang's identities
give value zero at the distinguished degenerate point $(1,0)$; along some unbounded degenerate
directions the value also approaches zero only polynomially. Verifying a
nonpositivity statement over this domain by interval arithmetic is the
obstruction that stopped a direct check.

We remove it by changing coordinates. The profile
$\Lambda_{\lambda_1,\lambda_2}$ is exactly the maximum-entropy profile for its
own moments (this is the Lagrange condition that produced the two-parameter
family in the first place), so
\[
  \ent(\Lambda_{\lambda_1,\lambda_2}) = H(a_1,a_2)
  := \max\bigl\{\, \E\,\ent_2(\tfrac{1+\Lambda}2) :
     \E[X\Lambda]=a_1,\ \E[M\Lambda]=a_2,\ |\Lambda|\le 1 \,\bigr\},
\]
and therefore
\begin{equation}\label{eq:moment-reduction}
\begin{aligned}
  \sup_{\lambda_1,\lambda_2} \mathcal{S}_\star(\lambda_1,\lambda_2)
  &\le \sup_{(a_1,a_2)\in K}
       \bigl[\, H(a_1,a_2) + G(a_1,a_2) \,\bigr],\\
  G(a_1,a_2)
  &= \inf_{s\ge0}\bigl\{\tfrac12 s^2\psi_0
     + \alpha_\star\, \E\log\Psi(V_s)\bigr\}.
\end{aligned}
\end{equation}
The key point is that the new domain
\[
  K = \bigl\{ (\E[X\Lambda],\E[M\Lambda]) : |\Lambda|\le 1 \bigr\}
\]
is the \emph{moment body} of the pair $(X,M)$, a compact convex set with
support function $h(u,v)=\sup_{|\Lambda|\le1}\E[\Lambda(uX+vM)]=\E|uX+vM|$. It
is contained in the rectangle $|a_1|\le\E|X|$, $|a_2|\le\E|M|$. The unbounded
tail of the $(\lambda_1,\lambda_2)$ plane is compressed onto $\partial K$,
where the profiles are saturated ($|\Lambda|=1$) and hence
$H\!\restriction_{\partial K}=0$.
Also $\E|M|<\sqrt{\E M^2}=\sqrt{q_0}$, so
$D=\sqrt{1-a_2^2/q_0}$ below is real and uniformly positive on $K$.

Two consequences make \eqref{eq:moment-reduction} checkable. First, the
constraint term $G$ depends on $a=(a_1,a_2)$ explicitly, through
\[
 \begin{aligned}
 V_s(a)&=
 \frac{-\frac{a_2}{q_0}\sqrt{q_0}\,Z
       -\frac{a_1}{\psi_0}\bN}{D}+s\bN,
 &D&=\sqrt{1-a_2^2/q_0},\\
 \bN&=\frac{\cE(-\gamma Z)}{\sqrt{1-q_0}},
 &\gamma&=\sqrt{q_0/(1-q_0)},
 \end{aligned}
\]
where $Z$ is standard Gaussian.  For the resulting one-dimensional Gaussian
integral, write
\[
 \mathcal T(a,s):=\E\log\Psi(V_s(a)).
\]
For every fixed $s$,
\[
 G(a)\le \tfrac12 s^2\psi_0 + \alpha_\star\mathcal T(a,s).
\]
Thus a per-cell choice of $s$ gives a certified upper bound with no inner
optimization. Second, convex duality gives, for \emph{every} dual point
$(b_1,b_2)$,
\begin{equation}\label{eq:dual-bound}
  H(a_1,a_2) \le \Phi(b_1,b_2) - b_1 a_1 - b_2 a_2,
  \qquad \Phi(b_1,b_2) = \E\log 2\cosh(b_1 X + b_2 M),
\end{equation}
with equality when $(b_1,b_2)$ is the dual of $(a_1,a_2)$. Choosing $(b_1,b_2)$
per cell (found by a Newton solve on the convex dual objective, then used as a
fixed rational) turns the entropy into an explicit linear function of
$(a_1,a_2)$, tight on the cell.

\begin{lemma}[Moment reduction]\label{lem:moment}
Write $\boldsymbol\lambda=(\lambda_1,\lambda_2)$ and let
$a(\boldsymbol\lambda)$ denote its two moment coordinates. With $H$ and $G$
as above,
\[
  \mathcal S_\star(\boldsymbol\lambda)
  =H(a(\boldsymbol\lambda))+G(a(\boldsymbol\lambda)),
  \qquad
  \sup_{\boldsymbol\lambda\in\R^2}\mathcal S_\star(\boldsymbol\lambda)
  \le\sup_{a\in K}(H(a)+G(a)),
\]
and \eqref{eq:dual-bound} holds for all $(b_1,b_2)$.
\end{lemma}
\begin{proof}
Fix $\boldsymbol\lambda=(\lambda_1,\lambda_2)$ and write
$\Lambda=\Lambda_{\lambda_1,\lambda_2}$,
$a_i=a_i(\boldsymbol\lambda)$. The constraint term of $\mathcal S_\star$ depends on
$\Lambda$ only through $(a_1,a_2)$ (it enters
$\mathcal S_\star(\Lambda,s)$ only via $\E[M\Lambda]=a_2$ and
$\E[X\Lambda]=a_1$, in the notation of \S\ref{sec:huang-moments}), so
$\mathcal S_\star(\lambda_1,\lambda_2)=\ent(\Lambda)+G(a_1,a_2)$. It remains
to show $\ent(\Lambda)=H(a_1,a_2)$, i.e. that $\Lambda$ maximizes the entropy
among all profiles with moments $(a_1,a_2)$. This is the Lagrange condition
that defines $\Lambda$: maximizing $\E\,\mathrm{ent}_2(\tfrac{1+\Lambda}2)$
over $|\Lambda|\le1$ subject to $\E[X\Lambda]=a_1,\E[M\Lambda]=a_2$ has,
by pointwise optimization of the Lagrangian
$\mathrm{ent}_2(\tfrac{1+\Lambda}2)+(\lambda_1 X+\lambda_2 M)\Lambda$, the
maximizer $\Lambda^\ast=\tanh(\lambda_1 X+\lambda_2 M)=\Lambda$. Hence
$\ent(\Lambda)=H(a_1,a_2)$ and
$\mathcal S_\star(\lambda_1,\lambda_2)
=(H+G)(a_1(\boldsymbol\lambda),a_2(\boldsymbol\lambda))$.
Taking suprema gives the stated one-sided inequality; no density assertion
about the finite-parameter image is needed. For \eqref{eq:dual-bound}, weak
duality (and strong duality in the interior) gives
$H(a)\le\inf_b[\Phi(b)-b\cdot a]$, where the inner
maximization of the Lagrangian over $|\Lambda|\le1$ gives
$\max_\Lambda[\mathrm{ent}_2(\tfrac{1+\Lambda}2)+\theta\Lambda]=\log2\cosh\theta$
at $\theta=b_1X+b_2M$; the stated inequality is the weak-duality direction,
valid for every $b$.
\end{proof}

The ray argument below differentiates $H$, so it also requires a finite dual
throughout the star.  This does not follow merely from compactness of $K$; we
certify a uniform interior neighborhood constructively.

\begin{lemma}[Uniform interior neighborhood]\label{lem:star-interior}
Uniformly over the certified interval for $\psi_0$, let $f=(X,\tanh X)$ and
$a^\star=\E[f\tanh X]=\nabla\Phi(1,0)$.  Then
\[
 B_2(a^\star,0.0150391982)\subset K.
\]
In particular, the full interval-inflated Region-I star, whose displacement
from the true $a^\star$ is at most
$0.012096$, lies in $\operatorname{int}K$.
For every point $a$ in that star there is a unique finite dual $b(a)$, and
\[
 H(a)=\Phi(b(a))-b(a)\cdot a,\qquad
 \nabla H(a)=-b(a),\qquad
 \nabla^2H(a)=-[\nabla^2\Phi(b(a))]^{-1}.
\]
\end{lemma}
\begin{proof}
Write $M=\tanh X$ and define
\[
 \rho_0(x)=\operatorname{sgn}(x)(1-|\tanh x|),\qquad
 \rho_1(x)=\rho_0(x)
 \begin{cases}1,&|x|\le3/4,\\-1,&|x|>3/4.
 \end{cases}
\]
If $|c_0|+|c_1|\le1$, then
$\Lambda_c=M+c_0\rho_0+c_1\rho_1$ satisfies
$|\Lambda_c|\le |M|+(1-|M|)(|c_0|+|c_1|)\le1$.  Hence
$a^\star+Ac\in K$, where $A_{ij}=\E[f_i\rho_j]$.  The certified uniform
enclosure is
\[
A\in\begin{pmatrix}
[0.15726546516,0.15726546571]&
[-0.01079401984,-0.01079401929]\\
[0.12551630837,0.12551630878]&
[0.01084146741,0.01084146783]
\end{pmatrix}.
\]
For its columns $p,q$, interval arithmetic gives
\[
 \begin{aligned}
 \det A&\ge0.003059813871082545,&
 \|p-q\|_2&\le0.2034559164094782,\\
 &&\|p+q\|_2&\le0.2001182846546655.
 \end{aligned}
\]
The image under $A$ of the $\ell^1$ unit diamond therefore contains a
Euclidean ball of radius
\[
 \frac{\det A}{\max\{\|p-q\|_2,\|p+q\|_2\}}
 >0.01503919829455.
\]
The widest long-radius angular leaf has width $0.016$.  Including its
$10^{-8}$ angular padding and a $10^{-10}$ Arb-rounding guard, the independently
enclosed sine and cosine have joint norm at most
$\sqrt{1+\sin(0.016+2\cdot10^{-8}+10^{-10})}$.  The radial padding is
$10^{-12}$, and the stored-center interval can be $2\cdot10^{-7}$ away
from the true center in each coordinate.  Thus every interval box on which
$H$ is differentiated is within
\[
 (0.012+10^{-12})\sqrt{1+\sin(0.016+2\cdot10^{-8}+10^{-10})}
 +2\sqrt2\,10^{-7}<0.012096
\]
of $a^\star$.  This is the required-radius inequality replayed by the
certificate.  It proves the inclusion and leaves more than $0.002943$
clearance beyond the padded star.

For any such $a$, choose $\varepsilon>0$ with
$B_2(a,\varepsilon)\subset K$.  If $u$ is a unit vector and $t\ge0$, then
\[
 \Phi(tu)-tu\cdot a
 \ge t\bigl(h_K(u)-u\cdot a\bigr)\ge t\varepsilon,
\]
because $\log(2\cosh y)\ge|y|$ and $h_K(u)=\E|u\cdot f|$.  Thus the dual
objective is coercive.  Its Hessian
$\E[ff^\top\operatorname{sech}^2(b\cdot f)]$ is positive definite: a
nonzero linear combination of $X$ and $\tanh X$ cannot vanish almost surely.
The minimizer is consequently finite and unique.  Pointwise entropy duality
gives the equality for $H$, and the inverse-function theorem gives the two
displayed derivative identities.
\end{proof}

Region~II certifies, cell by cell, the inequality
$\min_k[\Phi(b^k)-b^k\!\cdot a]+\tfrac12(s^C)^2\psi_0+\alpha_\star T_C(a)<0$ on
the cell, where $T_C$ encloses $\E\log\Psi(V_{s^C})$ over the cell by a
mean-value rule in $(a_1,a_2)$. By Lemma~\ref{lem:moment} the first term
bounds $H$ and the rest bounds $G$ from above on the cell, so the certified
inequality gives $(H+G)<0$ there. Cells whose closure misses $K$ are dropped:
a cell lies outside $K$ once some direction $(u,v)$ has
$\min_{a\in\mathrm{cell}}(u a_1+v a_2)>h(u,v)$, which is checked against a
certified upper bound on $h(u,v)=\E|uX+vM|$ over a fan of directions.

\section{The compact sweep and the degenerate point}\label{sec:huang-sweep}

Write $a^\star=(\psi_0(1-q_0),q_0)$ for the image of $(\lambda_1,\lambda_2)=(1,0)$;
the fixed-point identities give $H(a^\star)+G(a^\star)=0$, and the following
verification proves global maximality. We verify
\eqref{eq:moment-reduction} in two pieces.

\emph{Region II (the bulk).} For every cell $C$ of a bounded, adaptively
refined grid on $K$ away from the star region of Region~I, we certify
\[
  \bigl[\Phi(b^C) - b^C_1 a_1 - b^C_2 a_2\bigr]
  + \tfrac12 (s^C)^2\psi_0 + \alpha_\star\, T_C \;<\; 0
  \qquad\text{on } C,
\]
where $b^C,s^C$ are the (nonrigorous) per-cell dual and tilt and $T_C$ is a
mean-value enclosure of $\E\log\Psi(V_{s^C})$ over $C$. Cells lying outside
$K$ (detected by the support function $h$ along a fan of directions)
contribute nothing and are dropped; cells straddling $\partial K$ are handled
by pulling the dual inward, using that \eqref{eq:dual-bound} holds for any
fixed $b$. The sweep runs in two stages---all of $K$ minus a coarse box
around $a^\star$, then that box minus the star, bisecting to side $10^{-3}$
along the star boundary---and a cell is skipped only if a certified witness
places the full rectangle inside the core disk or inside one common signed
weak-axis cone by affine inequalities at all four corners. Both stages terminate
with every leaf cell strictly negative.

\emph{Region I (the maximizer).} Near the maximizer the value is genuinely
$0$ at $a^\star$, so no direct sign check can succeed there, and interval evaluation
of any Hessian of the problem over a two-dimensional box fails for a
structural reason: $\nabla^2\Phi(1,0)$ has condition number $\approx 82$
(the pair $(X,\tanh X)$ is strongly correlated), so the dual image of an
$a$-box is stretched by a factor $\approx 220$ along one direction and every
box enclosure of $\nabla^2 H=-[\nabla^2\Phi]^{-1}$ loses the cancellation it
needs. We work instead along rays.  The certificate fixes the single exact
rational vector
\[
 C=(-2.448324,\,0.790403)
\]
globally, and for every unit direction $v$ uses
$\dot\sigma_v=C\cdot v$ on every radial band of that ray.  Thus all bands
for a given $v$ concern the same function
\[
  \varphi_v(t) = H(a^\star+tv) + \tfrac12 s(t)^2\psi_0
     + \alpha_\star \mathcal T\bigl(a^\star+tv, s(t)\bigr),
  \qquad s(t)=s_0+t\dot\sigma_v,\quad s_0=\sqrt{1-q_0} .
\]
The implications used by the interval program are isolated in the next
lemma.  This makes explicit both the matrix-order direction and the
quantifiers in the dual-localization check.

\begin{lemma}[Ray and localization certificate]\label{lem:ray-localization}
Let $P\Subset\operatorname{int}K$ be one certified moment polygon, let
$L\subset\R^2$ be a closed bounded rectangle, and let
$\widehat b_1,\ldots,\widehat b_m$
lie in $\operatorname{int}L$.  For $a\in P$ set
$F_a(b)=\Phi(b)-b\cdot a$ and
\[
 S_a=\{b:F_a(b)\le\min_kF_a(\widehat b_k)\}.
\]
Suppose each certified subsegment $E\subset\partial L$ and each tangent
polygon $P_j$ covering $P$ has fixed weights
$w_{E,j,k}\ge0$, $\sum_kw_{E,j,k}=1$, for
which
\begin{equation}\label{eq:anchor-boundary}
 \sum_k w_{E,j,k}\{\Phi(b)-\Phi(\widehat b_k)
              -(b-\widehat b_k)\cdot a\}>0
 \quad (b\in E,\ a\in P_j).
\end{equation}
Then the entropy dual $b(a)$ lies in $\operatorname{int}L$ for every
$a\in P$.  Moreover, with $f=(X,\tanh X)$ define
\[
 B_L=\E\left[ff^\top\max_{b\in L}\operatorname{sech}^2(b\cdot f)\right],
\]
and let $\widehat B_L\succeq B_L$ be any certified positive-definite
majorant. Then, for every $v\in\R^2$,
\begin{equation}\label{eq:entropy-ray-bound}
 v^\top\nabla^2H(a)v\le-v^\top \widehat B_L^{-1}v.
\end{equation}
Consequently, fix a unit vector $v$, a radius $R(v)>0$, and the preceding
global choice $\dot\sigma_v=C\cdot v$.  Suppose
$0=t_0<t_1<\cdots<t_r=R(v)$ and that for every $i$ there are a certified
polygon $P_i\Subset\operatorname{int}K$ and a closed bounded rectangle
$L_i$ satisfying the preceding localization hypotheses, with
$\{a^\star+tv:t\in[t_{i-1},t_i]\}\subset P_i$.  Assume also throughout the
prefix that $a_2(t)^2<q_0$ and $s(t)\ge0$, and that on every segment the
certificate bounds the second derivative of the nonentropy part of
$\varphi_v$ by $v^\top \widehat B_{L_i}^{-1}v$.  Then $H+G\le0$ on the full ray prefix
$0\le t\le R(v)$.
\end{lemma}
\begin{proof}
By Lemma~\ref{lem:star-interior}, $F_a$ has the unique minimizer $b(a)$, and
$b(a)\in S_a$.  At least one anchor attaining the displayed minimum belongs
to $S_a\cap\operatorname{int}L$.  Condition~\eqref{eq:anchor-boundary}
implies
\[
 \max_k\{F_a(b)-F_a(\widehat b_k)\}
 \ge\sum_kw_{E,j,k}\{F_a(b)-F_a(\widehat b_k)\}>0
\]
on $\partial L$, so $S_a$ misses the boundary.  Convexity of $S_a$ and a
line-segment argument from its interior anchor show that $S_a\subset
\operatorname{int}L$.

The actual boundary sweep proves \eqref{eq:anchor-boundary} in exact
mean-value form.  On a segment of half-length $h$ and midpoint $b_0$, its
lower bound is the midpoint value minus
$h\sup_E|\partial_u\Phi(b)-a_u|$.  For fixed weights this lower bound is a
concave function of $a$ (affine minus a positive multiple of an absolute
value), so strict positivity at every vertex of the tangent polygon implies
strict positivity throughout it.  The same weights are retained over the
whole boundary subsegment and polygon; they are never selected separately at
different vertices.

For $b(a)\in L$,
\[
 \nabla^2\Phi(b(a))\preceq B_L\preceq\widehat B_L.
\]
These matrices are positive definite, and inversion reverses Loewner order:
\[
 \widehat B_L^{-1}\preceq B_L^{-1}
 \preceq[\nabla^2\Phi(b(a))]^{-1}.
\]
Lemma~\ref{lem:star-interior} now gives \eqref{eq:entropy-ray-bound}.  On
each $[t_{i-1},t_i]$ the asserted nonentropy bound therefore gives
$\varphi_v''\le0$.  The radial pieces form a gapless partition of
$[0,R(v)]$, and the same fixed slope $\dot\sigma_v=C\cdot v$ is used on
every piece, so these are restrictions of one continuously differentiable
function rather than unrelated annular majorants.  Hence $\varphi_v'$ is
nonincreasing on the entire prefix and $\varphi_v$ is concave there.
Finally, $G$ is the infimum over the tilt, so $H+G\le\varphi_v$, while
Huang's exact fixed-point identities give
$\varphi_v(0)=\varphi_v'(0)=0$.  It follows that
$\varphi_v(t)\le0$, and therefore $H+G\le0$, for every
$0\le t\le R(v)$.
\end{proof}

The matrix $B_L$ is a one-dimensional integral whose inner maximum is a
per-$z$ corner-or-zero selection. In the positive-root branch the implementation
replaces the true zero-straddling strip by the certified outer strip
$|z|\le z_q$ and uses weight one there, producing the Loewner majorant
$\widehat B_L\succeq B_L$;
in the zero-root branches $\widehat B_L=B_L$. The global entropy bound
$B_{\R^2}=\E[ff^\top]$ settles all but a narrow angular window around the weak
eigenvector; inside it the certified bounded boxes become sufficiently tight
around $(1,0)$ for $\widehat B_L$ to retain the sharpness needed by the certificate.

It remains to match the polar certificates to the stage-2 skips.  Let
$\omega(\theta)$ be angular distance to either weak-eigenvector axis.  The
stage-2 test uses the deliberately shrunken effective radii
\[
 T_{\rm eff}(\theta)=
 \begin{cases}
 0.012,&\omega(\theta)\le0.1479995,\\
 0.008,&\omega(\theta)\le0.4199995,\\
 0.005,&\text{otherwise}.
 \end{cases}
\]
The stage-2 inclusion witness encloses the maximum corner radius in Arb,
including the certified center uncertainty. It either accepts the rectangle
inside the everywhere-valid $0.005$ disk, or selects one common signed weak
axis and verifies positive projection and both affine half-cone inequalities
at all four corners. The cone half-angles are shrunk by $0.012$ and $0.030$
and, like every radius, by a further $5\cdot10^{-7}$. These angular
shrinkages exceed the adjacent Region-I chunk half-widths $0.008$ and
$0.025$. Convexity then puts the entire
rectangle inside the corresponding shrunken disk or cone. Thus every skipped
rectangle lies in the union of accepted Region-I polar cells, so the two
regions have neither a geometric gap nor an unchecked boundary sliver. All of
these checks use the full Ding--Sun parameter balls.

\section{Assembly}\label{sec:assembly}

The pieces combine as follows.

\emph{Upper bound.} Huang's upper-bound theorem~\cite{huang2024capacity}
reduces the capacity upper bound at $\kappa=0$ to four numerical conditions.
His interval code verified the AMP and local-concavity conditions. The
parameter rectangle was inherited from Ding--Sun Proposition~1.3 and is
re-verified here in \S\ref{sec:gardner} (Block~1). His remaining
Condition~1.3 is the statement
$\mathcal S_\star\le 0$ on $\R^2$. In moment coordinates
(\S\ref{sec:huang-moments}) the one-sided reduction shows that
$\sup_{K}(H+G)\le 0$ is sufficient; the two sweeps of
Region~II (\S\ref{sec:huang-sweep}) certify $H+G<0$ on $K$ outside the
certified star region around $a^\star$ (the first over $K$ minus a coarse
box, the second over that box minus the star, bisecting to side $10^{-3}$
along the star boundary), and Region~I's ray certificates give $H+G\le0$ on
the star itself. Hence $\sup_K(H+G)=0$, and the one-sided reduction gives
Condition~1.3. Huang's upper bound
$\lim_N \PP(M_N/N\ge\alpha)=0$ for $\alpha>\alpha_\star$ becomes
unconditional.

\emph{Lower bound.} Under their Condition~1.2, Ding--Sun's main theorem gives
\[
 \liminf_{N\to\infty}\PP(M_N/N\ge\alpha)>0
 \qquad (\alpha<\alpha_\star).
\]
Nakajima--Sun's theorem for general subgaussian disorder upgrades this to high
probability; it extends Xu's Bernoulli-model result. Writing
$\mathrm{PG}=\mathcal H+\mathcal P$, the global-sign clause of
Condition~1.2 is $\mathscr S_\star(\lambda)<0$ for
$\lambda\notin\{0,1\}$, and the covering of
$[\lambda_{\min},1]$ follows their own decomposition: the value grids on
$[0.2,0.982]$ and $[\lambda_{\min},-0.125]$ (part (a), our Block~3a), the
derivative signs $\mathrm{PG}'<0$ on $[0.05,0.2]$ and $\mathrm{PG}'>0$ on
$[-0.125,-0.03]$ (part (b)), the second derivative $\mathrm{PG}''<0$ on
$[-0.03,0.05]$ around the degenerate zero $\lambda=0$ where
$\mathrm{PG}(0)=\mathrm{PG}'(0)=0$ (part (c)), and the near-one analysis on
 $[0.982,1)$ (Block~2); the value grid reaches beyond $0.982$, so the two
 pieces overlap. The central certificate and $\mathcal A''(0)\le0$
 \cite[discussion following eq.~(2.35)]{ding2018capacity} give the
 separate journal-version clause $\mathscr S_\star''(0)<0$; Ding--Sun's exact
 calculation \cite[proof of Theorem~1.4]{ding2018capacity} gives
 $\mathscr S_\star'(\lambda)\to+\infty$ as
 $\lambda\uparrow1$, which is stronger than their endpoint clause. Parts (b)
 and (c) use the exact derivative identities
$\mathcal H'=-(1-q_\star)\log A/2$, $\mathcal
H''=-(1-q_\star)/(2A\,\ell'(A))$ and certified two-dimensional integrals for
$I'$ and $I''$ (with closed-form Gaussian-moment bounds for the mass outside
the integration box); the interval endpoints are pinned by the certified
monotone map $\lambda=\ell(A(\tau))$, exactly as in part (a). On the negative
interval of part (b) the margins are uniformly thin (about $2\times10^{-3}$)
while a ball-parameter evaluation of $I'$ wraps the cell width by two orders
of magnitude, so the cells there are certified in mean-value form: $I'$ at
the exact cell center, plus a bound on $|\mathrm{PG}''|$ times the cell
radius, the latter assembled from certified point values of $I''$ on a fine
grid together with a closed-form bound on $|I'''|$ (a degree-four
Gaussian-moment estimate) that controls the values between grid points.

\emph{Disorder.} Krauth and M\'ezard's model has $\pm1$ (Bernoulli) disorder;
we work with Gaussian disorder throughout. Nakajima--Sun show that the
sharp-threshold locations differ by $o(1)$ across the standardized subgaussian
class; hence the Gaussian convergence transfers to Bernoulli disorder and the
other distributions stated in Theorem~\ref{thm:main}.

\emph{What is machine-checked.} Every numerical certificate invoked above is
a finite list of ball inequalities between explicitly defined real numbers,
established by Arb; the analytic reductions around them (the moment-body reduction, convex
duality, the ray majorants and their pinned identities, the localization and
$B_L\preceq\widehat B_L$ matrix arguments, and the reductions to certified quadrature) are proved on
paper. The nonrigorous companion (\texttt{huang\_np.py}) only selects
per-cell duals, tilts, anchors, and candidate boxes, which enter the
certificates as fixed rationals; no certified inequality depends on it.
For Region~I, the final verifier does not merely inspect the signs of stored
balls: from the frozen source it recomputes every fixed-weight localization
witness, every entry of $\widehat B_L$ and its determinant, every inverse quadratic
form, and every radial curvature enclosure, then requires byte-exact packet
agreement.  The top-level verification also binds the complete angular and
radial trees, the four-artifact Region-I/II delegation chain, and the pinned
Python executable and python-flint arithmetic core.

\emph{Inventory.} The final source-bound runs comprise: the parameter rectangle, 13 checks; Huang Region~II stage~1, $1200$ top cells and $1822$ verified leaves; stage~2, $240$ top cells and $1499$ verified leaves; Huang Region~I, $1404$ band jobs, $4542$ certified angular leaves, and $16543$ radial pieces; Ding--Sun Block~3a, $247$ top cells and $263$ recursive leaves; the corrected near-one block; and Block~3b/c at $K_{\mathrm{run}}=21/2$, with $24$ positive-branch top cells ($124$ leaves), $331$ negative-branch top cells ($578$ leaves), and $16$ central top cells ($185$ leaves), supported by a certified $59$-point $I''$ grid. The canonical JSON certificates and their complete raw record trees are included with the source.

\begin{remark}[Soundness audit]
Two implementation errors were caught during this work by independent
cross-checks and corrected before the final runs; we record them because
they illustrate what interval arithmetic does and does not protect against.
First, an early version of the fixed-grid quadrature multiplied the
derivative enclosure of the mean-value remainder by the \emph{signed} first
moment $\int_{\text{cell}}(z-m)\varphi\,dz\approx 0$; since the enclosed
derivative varies over the cell while $(z-m)$ changes sign, the remainder
does not factor through the signed moment, and the resulting balls were
systematically too tight---detected because a 30-digit quadrature of
$\nabla\Phi(1,0)$ fell outside them. The corrected rule handles the
positive and negative parts of $(z-m)$ separately. Second, a sign error in
$\partial_s^2$ of the tilted constraint term ($+$ for $-$) survived until a
finite-difference cross-check of every ingredient of $\varphi_v''$ against
an independent floating-point implementation. Neither error is visible to
the interval machinery itself: Arb guarantees the arithmetic, not the
formulas. All certificates reported here postdate both corrections.
\end{remark}

\section*{Funding and declarations}
The author acknowledges support from the European Research Council under the
European Union's Horizon Europe programme (grant 101041711), the Simons
Foundation, Heights Labs, and Israel Science Foundation grants 2258/19 and
4101/25.

Most of the work reported here was done by two AI systems: Fable~5
(Anthropic) and Codex (OpenAI). Fable~5 developed the main mathematical
architecture---the moment-coordinate reduction, the ray majorants and their
pinned identities, the $B_L\preceq\widehat B_L$ bound, the sublevel-set
localization, and the corrected near-one chain---and wrote the bulk of the
proof code and the initial manuscript. Codex performed a separate internal
re-audit of the final source and proof boundary, applied and validated the confirmed
full-circle normalization fix in the Region~I verifier, validated and landed
the repaired Block~3a and Huang closures, completed the fresh Block~3b/c and
final closure, hardened the source-bound receipt and cross-platform release
gates, assembled the final reproducibility package, and audited the final
manuscript and public-release readiness.

Both systems operated in autonomous goal modes: the author set the objective
of resolving the outstanding conditions, and the systems worked toward that
objective over extended unattended sessions, choosing approaches,
implementing and running certificates, detecting and repairing errors through
separate cross-checks, and iterating until the verifications closed. The
author supervised the runs, reviewed the mathematics, and accepts
responsibility for the final manuscript. This is
part of a broader research project of the author on the extent to which
current AI systems can resolve open problems in mathematics.

\end{document}